\documentclass[11pt]{amsart}
\usepackage{amsfonts,amsmath,amssymb}

\def\ot{\otimes}

\def\rhu{\hbox{$\rightharpoonup$}}
\def\lhu{\hbox{$\leftharpoonup$}}

\newtheorem{definition}{Definition}
\newtheorem{proposition}[definition]{Proposition}
\newtheorem{corollary}[definition]{Corollary}

\newtheorem{remarks}[definition]{Remarks}
\newtheorem{theorem}[definition]{Theorem}

\def\bea{\begin{eqnarray*}}
\def\eea{\end{eqnarray*}}
\def\veps{\varepsilon}

\def\ot{\otimes}

\def\a{$\mbox{\u{a}}$}

\def\S{$\mbox{\c{S}}$}

\def\veps{\varepsilon}

\def\bea{\begin{eqnarray*}}
\def\eea{\end{eqnarray*}}

\def\ot{\otimes}

\begin{document}

\title{The bijectivity of the antipode revisited}
%\dedicatory{}
%\author{M. Beattie}
%\address{Department of Mathematics and Computer Science, Mount Allison University,
%Sackville, New Brunswick, Canada E4L 1E6 }
%\email{mbeattie@mta.ca}
%\urladdress{http://www.mta.ca/~mbeattie/}
\author{M.C. Iovanov}
%\thanks{The second author was partially supported by the contract nr. 24/28.09.07 with UEFISCU ``Groups, quantum groups, corings and representation theory" of CNCIS, PN II (ID\_1002)}
\address{University of Southern California, 3600 Vermont Ave. \\
Los Angeles, CA 90089, USA, and 
\\University of Bucharest, Fac. Matematica \& Informatica,
Str. Academiei nr. 14,
Bucharest 010014,
Romania}
\email{yovanov@gmail.com, iovanov@usc.edu}
%\urladdress{http://yovanov.net}
\author{\c{S}. Raianu}
\address{Mathematics Department, California State University, Dominguez Hills,
1000 E Victoria St, Carson CA 90747, USA}
\email{sraianu@csudh.edu}
%\urladdress{http://www.csudh.edu/math/sraianu}
\begin{abstract} We provide a very short approach to several fundamental results for Hopf algebras with nonzero integrals. Besides being short, our approach is the first to prove the bijectivity of the antipode without using the uniqueness of the integrals of Hopf algebras and to obtain the uniqueness of integrals as a corollary in a way similar to the classical theory of the Haar measure on compact groups.
%Proof of the bijectivity of the antipode of co-Frobenius Hopf algebras without using the uniqueness of integrals.
\end{abstract}
\date{May 15, 2010}
\maketitle

\section*{Introduction }

 One of the fundamental notions of the theory of Hopf algebras is that of an integral, which is an analog of the Haar integral of a compact group and draws its name from there. %An integral $\lambda$ of a Hopf algebra $H$ is an element in $H^*$ such that $\alpha\lambda=\alpha(1)\lambda$ for all $\alpha\in H^*$. 
More precisely, if $G$ is a compact group and $R(G)$ is the algebra of continuous representative functions on $G$, i.e. the space spanned by the coefficients $\eta_{ij}$ of all continuous representations $\eta:G\rightarrow GL_n(\mathbb{C})$, then the restriction of the Haar integral to $R(G)$ becomes an integral in Hopf algebra sense. In this respect, a Hopf algebra having a nonzero integral is a generalization of the algebra of continuous representative functions on a compact group. Integrals for Hopf algebras were introduced by Sweedler in 1969 \cite{sw}. In that paper he proves a series of fundamental results about Hopf algebras with nonzero integrals, including the fact that integrals are unique and the antipode is bijective when the Hopf algebra is finite dimensional. The questions about the validity of these results for Hopf algebras with nonzero integrals of possibly infinite dimension appear explicitly in Sweedler's book \cite{sw1}.  These questions were given affirmative answers: the  uniqueness of the integrals was proved by Sullivan in 1971 \cite{sul}, then, using Sullivan's result, Radford proved in 1977 that the antipode of a Hopf algebra with nonzero integrals is bijective \cite{rad}. Many other proofs for the uniqueness of integrals were found later. Some of these proofs have a strong homological flavor and use the fact that integrals are just comodule maps \cite{st}, \cite{bdgn}, \cite{mtw}, \cite{dnt}. In contrast with the abundance of proofs for the uniqueness of integrals, Radford's proof for the bijectivity of the antipode was virtually the only one available (with a simplification due to \cite{cal}) until very recently, when  
%It turned out that the existence of an integral is equivalent to various equivalent representation theoretic properties of the Hopf algebra (see \ref{cor}), such that of being co-Frobenius as a coalgebra, or having nonzero rational part. Quite interestingly, once a nonzero integral exists, it is unique and the antipode of the Hopf algebra is bijective, as proved by Radford \cite{rad}. 
alternate proofs were obtained in \cite{i} and \cite{i0} by using a purely coalgebraic approach, as a byproduct of a general theory of algebraic ``integrals'' or infinite dimensional generalized Frobenius algebras. All proofs for the bijectivity of the antipode used the uniqueness of integrals, and it was hard to say whether this happened by necessity or it was just an effect of the order in which the two results were obtained. Moreover, the classical proof for the uniqueness of Haar measures adapted for Hopf algebras requires the bijectivity of the antipode (see \cite{vd} and \cite{r}). In the classical case of compact groups, the Hopf algebra of representative continuous functions clearly has a bijective antipode because it is commutative, and this probably made the causative relationship between bijectivity and uniqueness harder to understand. The fact that the antipode of a Hopf algebra with nonzero integrals might not be necessarily bijective was the only obstacle in proving the uniqueness of the integrals by using the same technique as in the case of Haar measures. 

In this note we find a very short approach to explain the above mentioned results. We first prove the bijectivity of the antipode without using the uniqueness of the integrals. This is the first proof constructed in this manner, and it follows by using a technique from \cite{i}. We can then just use the classical proof of the uniqueness of the Haar measure from locally compact groups, as was done in \cite{vd} for multiplier Hopf algebras (see also the chapter on Haar measures in \cite{bour}). Thus, besides being short, this proof also has the advantage that it shows once more an even stronger parallel than noted previously between Hopf algebras and locally compact groups. %used for a proof of the bijectivity of the antipode.

\section*{The Proofs}

Let $H$ be a Hopf algebra over the field $k$. Recall that an integral $\lambda$ of the Hopf algebra $H$ is an element in $H^*$ such that $\alpha\lambda=\alpha(1)\lambda$ for all $\alpha\in H^*$.

For $(M,\rho)\in {\mathcal M}^H$, 
$$\rho:M\longrightarrow M\otimes H,\;\;\;\rho(m)=m_0\otimes m_1,$$ 
we define ${^S\!{M}}\in {^H\!{\mathcal M}}$ with comodule structure given by 
$$m\longmapsto m_{(-1)}\otimes m_{(0)}=S(m_1)\otimes m_0$$ 
It is clear that we have a functor $F:{\mathcal M}^H\longrightarrow {^H\!{\mathcal M}}$, $F(M)={^S\!{M}}$, and $F$ is the identity on morphisms. Then we have:
\begin{proposition}\label{srathopfmod}
${^S\!{Rat(H^*)}}$, with left $H$-module structure given by
$$H\otimes {^S\!{Rat(H^*)}}\longrightarrow {^S\!{Rat(H^*)}},\;\;\;x\otimes\alpha\longrightarrow x\rhu\alpha$$
and left $H$-comodule structure as above is a left $H$-Hopf module.
\end{proposition}
\begin{proof}      
We need to prove that
$$(x\rhu\alpha)_{(-1)}\otimes(x\rhu\alpha)_{(0)}=x_1\alpha_{(-1)}\otimes x_2\rhu\alpha_{(0)}$$
which is
$$S((x\rhu\alpha)_1)\otimes(x\rhu\alpha)_0=x_1S(\alpha_1)\otimes x_2\rhu\alpha_0$$
or
$$<\beta S((x\rhu\alpha)_1)(x\rhu\alpha)_0,y>=<\beta(x_1S(\alpha_1))(x_2\rhu\alpha_0),y>,\;\;\;\forall\beta\in H^*,\;\; y\in H$$
%(Note: the left $H$-module structure on ${^S\!{Rat(H^*)}}$ can also be obtained as follows: let $H^*_S$ be the right $H^*$-module structure of $H^*$ given by $h^*\cdot\alpha=(\alpha\circ S)*h^*=S^*(\alpha)*h^*$ (*=convolution). Then ${^S\!{Rat(H^*)}}=Rat(H^*_S)$, but I don't think this requires the bijectivity of $S$, here's why: when showing that $Rat(H^*_S)=Rat(H^*)$ we only use the fact that $S$ is injective (i.e. it has a left inverse) to prove the inclusion $Rat(H^*_S)\subseteq Rat(H^*)$, nothing is needed for the other inclusion; once this is done, we only need to describe the left $H$-comodule structure in terms of the right $H$-comodule structure, and there only $S$ is used, and not its inverse.)\\
We have
\bea
<\beta S((x\rhu\alpha)_1)(x\rhu\alpha)_0,y>&=&<(\beta\circ S)*(x\rhu\alpha),y>\;\;\;\;\mbox{(rt $H$-com str of $Rat(H^*))$}\\
&=&\beta S(y_1)(x\rhu\alpha)(y_2)\\
&=&\beta S(y_1)\alpha(y_2x)\\
&=&\beta S(y_1)\alpha(y_2x_2)\veps(x_1)\\
&=&\beta(\veps(x_1)S(y_1))\alpha(y_2x_2)\\
&=&\beta(x_1S(x_2)S(y_1)\alpha(y_2x_3)\\
&=&(\beta\lhu x_1)(S(y_1x_2))\alpha(y_2x_3)\\
&=&<(\beta\lhu x_1)\circ S,(yx_2)_1><\alpha,(yx_2)_2>\\
&=&<((\beta\lhu x_1)\circ S)*\alpha,yx_2>\\
&=&<((\beta\lhu x_1)\circ S)(\alpha_1)\alpha_0,yx_2>\\
&=&\beta(x_1S(\alpha_1))\alpha_0(yx_2)\\
&=&<\beta(x_1S(\alpha_1))(x_2\rhu \alpha_0),y>,
\eea
which ends the proof.
\end{proof}
%Now we can use the following simple observation (there are several ways to proceed from this point):\\
Let $C$ be a coalgebra and $M\in{^C\!{\mathcal M}}$. The coalgebra $C_M$ associated to $M$ is 
%defined as the subspace of $C$ spanned by all the $c_j$'s coming from:
%$$\rho(m)=\sum_{i=1}^n c_i\ot m_i\in C\ot M,\;\; m\in M.$$
%Alternatively, it can be defined as 
the smallest subcoalgebra $C_M$ of $C$ such that $\rho(M)\subseteq C_M\ot M$, i.e. $\displaystyle{C_M=\cap_{A\subseteq C,\rho(M)\subseteq A\ot M}}A$ (see \cite[p. 102]{dnr}). With this notation we have:
\begin{proposition}\label{inclcoalgcomod}
If $M\stackrel{f}{\twoheadrightarrow}N$ is a surjective morphism of left $C$-comodules, then $C_N\subseteq C_M$.
\end{proposition}
\begin{proof} 
%Proof 1: Let $c\in C_N$, $c=c_1$, $n\in N$, $\displaystyle{\rho_N(n)=\sum_{i=1}^kc_i\ot n_i}$ with the $c_i$'s and $n_i$'s linearly independent. Let $m\in M$ such that $f(m)=n$, so 
%$$\displaystyle{\sum_{i=1}^kc_i\ot n_i=\rho_N(f(m))=m_{(-1)}\ot f(m_{(0)})\in C\ot N}.$$ 
%Pick $\{n^*_i\}$ local basis dual to $\{n_i\}$, and get 
%$$$\displaystyle{c=c_1=\sum_{i=1}^kc_i n^*_i(n_i)=m_{(-1)} n^*_i(f(m_{(0)}))\in C_M.}$$
%Since $C_N$ is spanned by elements like $c$ above, we get that $C_N\subseteq C_M$.\\
%Proof 2. (I prefer this one) 
Let $K=Ker(f)$. Then $\rho_M(K)=\rho_K(K)\subseteq C\ot K$. Also, since $\rho_M(M)\subseteq C_M\ot M$, it follows that $\rho_M(K)\subseteq C_M\ot M$. So $\rho_M(K)\subseteq (C_M\ot M)\cap(C\ot K)$, i.e. $\rho_K(K)=\rho_M(K)\subseteq C_M\ot K$. Therefore, $K$ is a $C_M$-submodule, so $\rho_{M/K}(M/K)\subseteq C_M\ot M/K$, i.e. $\rho_N(N)\subseteq C_M\ot N$, hence $C_N\subseteq C_M$ by definition.
\end{proof}
%Note that $C_{\bigoplus_{i\in I}M_i}=\sum_{i\in I}C_{M_i}$ (straightforward).
We are now ready to prove
\begin{theorem}\label{antipodebijective}
If $H$ is co-Frobenius, $S$ is bijective.
\end{theorem}
\begin{proof}
By Proposition \ref{srathopfmod} and the fundamental theorem of Hopf modules, we have that ${^S\!{Rat(H^*)}}\simeq H\ot(^S\!{Rat(H^*)})^{co}=H\ot\int_l$, since it is easy to see that $(^S\!{Rat(H^*)})^{co}=\int_l$. Also, since $(Rat(H^*)^H\simeq\int_l\ot H=H^{(dim\int_l)}$ in ${\mathcal M}^H$, we get ${^S\!{Rat(H^*)}}\simeq(^S\!H)^{(dim\int_l)}=\bigoplus_{dim\int_l}{^S\!{H}}$, using the fact that the functor $F$ clearly commutes with direct sums. Since $Rat(H^*)\ne 0$ (equivalently $\int_l\ne 0$) we can find a surjection of left $H$-comodules
$$\pi:(^S\!H)^{(dim\int_l)}\simeq {^S\!{Rat(H^*)}}\simeq H\ot(^S\!{Rat(H^*)})^{co}\twoheadrightarrow {^H\!{H}}$$
%Look at $\pi:(^S\!H)^{(dim\int_l)}\twoheadrightarrow {^H\!{H}}$, t
Then $C_H\subseteq\sum C_{{^S\! H}}= C_{{^S\! H}}$ by Proposition \ref{inclcoalgcomod} and the obvious fact that $C_{\bigoplus_{i\in I}M_i}=\sum_{i\in I}C_{M_i}$. Obviously, $C_H=H$ (by the counit property), and also $C_{{^S\! H}}=S(H)$, since $\forall h\in H$, $S(h)=S(h_2)\veps(h_1)\in C_{{^S\! H}}$, because $\rho_{{^S\! H}}(h)=S(h_2)\ot h_1\in H\ot{^S\! H}$. So $H\subseteq S(H)$, and the proof is complete.
\end{proof}
\begin{corollary}\label{sintnz}
If $t\in\int_l$, $t\ne 0$, then $t\circ S\in\int_r$, $t\circ S\ne 0$.
\end{corollary}
\begin{proof}
Obvious.
\end{proof}
As a consequence of this proof, the proof for the uniqueness of integrals can be translated verbatim to Hopf algebras from the case of Haar measures, as was done by Van Daele in \cite{vd} for regular multiplier Hopf algebras. This proof could not be used for Hopf algebras because it requires the bijectivity of the antipode, and until now all proofs of the bijectivity of the antipode used the uniqueness of integrals. A modified version of the proof below not requiring the bijectivity of the antipode was given in \cite{r}.
\begin{corollary}
The dimension of $\int_l$ is at most one. 
\end{corollary}
\begin{proof} (Identical to the proof of \cite[Theorem 3.7]{vd})
Let $t_1,t_2\in\int_l$, $t_2\ne 0$. By Corollary \ref{sintnz} $\lambda=t_2\circ S
\in\int_r\setminus\{0\}$. Then for any $h\in H$ there is a $g\in H$ such that 
\begin{eqnarray}\label{t1t2}
t_1(xh)=t_2(xg)\;\;\;\forall x\in H.
\end{eqnarray}
Indeed, let $l,m\in H$ such that $\lambda(l)=1$ and $t_2(m)=1$. Then 
\bea
t_1(xh)&=&\lambda(l)t(xh)\\
&=&\lambda(x_1h_1l)t_1(x_2h_2)\;\;\;(\textstyle{t_1\in\int_l})\\
&=&\lambda(x_1h_1l_1)t_1(x_2h_2l_2S(l_3))\\
&=&\lambda(xhl_1)t_1(S(l_2))\;\;\;(\textstyle{\lambda\in\int_r})\\
&=&\lambda(xe)\;\;\;(e=hl_1t_1(S(l_2))\\
&=&\lambda(xe)t_2(m)\\
&=&\lambda(x_1e_1)t_2(x_2e_2m)\;\;\;(\textstyle{\lambda\in\int_r})\\
&=&\lambda(x_1e_1m_2S^{-1}(m_1))t_2(x_2e_2m_3)\\
&=&\lambda(S^{-1}(m_1))t_2(xem_2)\;\;\;(\textstyle{t_2\in\int_l})\\
&=&t_2(xg)\;\;\;(g=em_2\lambda(S^{-1}(m_1)).\\ 
\eea
We finish the proof by showing that $t_1$ is a scalar multiple of $t_2$. For $y\in H$ we have:
\bea
t_1(y)&=&\lambda(l)t_1(y)\\
&=&\lambda(l_1)t_1(yl_2)\;\;\;(\textstyle{\lambda\in\int_r})\\
&=&\lambda(S(y_1)y_2l_1)t_1(y_3l_2)\\
&=&\lambda(S(y_1))t_1(y_2l)\;\;\;(\textstyle{t_1\in\int_l})\\
&=&\lambda(S(y_1))t_2(y_2g)\;\;\;(\ref{t1t2})\\
&=&\lambda(e)t_2(y),
\eea
where the last equality follows from reversing the previous three equalities, and the proof is complete.
\end{proof}
\begin{remarks}
a) Note that aside from the bijectivity of the antipode, the proof above uses only the definition of integrals.\\ 
b) To compensate for not being able to use the inverse of $S$, a special left integral had to be chosen in \cite{r} and it was shown to form a basis of $\int_l$. Once we are able to use the inverse of $S$, the proof above shows that any non-zero left integral will do.
\end{remarks}
Following the work of Lin, Larson, Sweedler, and Sullivan, the existence of a nonzero integral is equivalent to various representation theoretical properties of the Hopf algebra, such as that of being co-Frobenius as a coalgebra, or having nonzero rational part. As a final application, we show how our approach may be used to simplify the proof of some of these results (see \cite[Theorem 5.3.2]{dnr}):
\begin{corollary}\label{cor}
For a Hopf algebra $H$ the following assertions are equivalent:\\
(1) $H$ is left co-Frobenius\\
(2) $H$ is left quasi-co-Frobenius\\
(3) $H$ is left semiperfect\\
(4) $Rat(H^*)\ne 0$ ($H^*$ as a left $H^*$-module)\\
(5) $\int_l\ne 0$\\
(6) The right hand version of (1)-(5)
\end{corollary}
\begin{proof}
All implications follow directly from the definitions or from Sweedler's isomorphism, with the exception of 5)$\Rightarrow$6) which follows from Corollary \ref{sintnz}.
\end{proof}

%\section{Preliminaries}

\end{document}